\documentclass[12pt]{amsart}

\usepackage{geometry}
\usepackage{graphicx}
\usepackage{latexsym}
\usepackage{amssymb}
\usepackage{amsmath}
\usepackage{amscd}
\usepackage{amsthm}
\usepackage{url}
\usepackage{hyperref}

\renewcommand{\bar}{\overline}

\renewcommand{\phi}{\varphi}

\newcommand{\spec}{\operatorname{Spec}}

\newcommand{\nm}{\operatorname{Nm}}
\newcommand{\disc}{\operatorname{Disc}}
\newcommand{\gal}{\operatorname{Gal}}

\newcommand{\kgal}{K^\text{gal}}

\newcommand{\zz}{\mathbb{Z}}
\newcommand{\qq}{\mathbb{Q}}

\newcommand{\oo}{\mathcal{O}}

\newtheorem{thm}{Theorem}[section]
\newtheorem{lm}[thm]{Lemma}
\newtheorem{prop}[thm]{Proposition}

\theoremstyle{remark}

\newtheorem*{rem*}{Remark}

\begin{document}
\title{Progress Towards Counting $D_5$ Quintic Fields}
\thanks{The authors are grateful for the support of the NSF in funding the Emory 2011 REU. The authors would like to thank our advisor Andrew Yang, as well as Ken Ono for their guidance, useful conversations, improving the quality of exposition of this article, and hosting the REU}

\author{Eric Larson}
\address{Department of Mathematics. Harvard University, Cambridge, MA 02138.}
\email{elarson3@gmail.com}

\author{Larry Rolen}
\address{Department of Mathematics and Computer Science, Emory University, Atlanta, GA 30322.}
\email{larry.rolen@mathcs.emory.edu}

\begin{abstract}
Let  $N(5,D_5,X)$ be the number of quintic number fields whose
Galois closure has Galois group
$D_5$ and whose discriminant is bounded by
$X$. By a conjecture of Malle, we
expect that $N(5,D_5,X)\sim C\cdot X^{\frac{1}{2}}$ for some constant
$C$.
The best known upper bound is $N(5,D_5,X)\ll X^{\frac{3}{4}+\varepsilon}$, and we
show this could be improved by counting points on a certain
variety defined by a norm equation;
computer calculations give strong evidence that this number is
$\ll X^{\frac{2}{3}}$. Finally, we show how such norm equations
can be helpful by reinterpreting an earlier proof of Wong on
upper bounds for $A_4$ quartic fields in terms of a similar norm equation.
\end{abstract}
\maketitle
\section{Introduction and Statement of Results}
Let $K$ be a number field and $G\leq S_n$ a transitive permutation
group on $n$ letters. In order to study the distribution of fields with
given degree and Galois group, we introduce the following counting function:
\[N(d,G,X):=\#\{\text{degree $d$ number fields $K$ with $\gal(\kgal / \qq) \simeq G$ and $\vert D_K\vert\leq X$}\}.\] Here $D_K$ denotes the discriminant of $K$, counting
conjugate fields as one.
Our goal is to study this function for
$d=5$ and $G=D_5$. In \cite{Malle}, Malle conjectured that
\begin{equation}
N(d,G,X)\sim C(G)\cdot X^{a(G)}\cdot\log(X)^{b(G)-1}
\end{equation}
for some constant $C(G)$ and for explicit constants $a(G)$ and $b(G)$, and this has been proven for all
abelian groups $G$. Although this conjecture seems to be close to the truth on the whole, Kl\"{u}ners found a counterexample
when $G=C_3\wr C_2$ by showing that the conjecture predicts the wrong value for $b(G)$ in \cite{Kluners-Counter}. This conjecture has been modified to explain all known counter-examples in \cite{Turkelli}.
\\ \indent We now turn to
the study of $N(5,D_5,X)$. By Malle's conjecture, we expect
that
\begin{equation}
N(5,D_5,X)\buildrel{?}\over{\sim}C\cdot X^{\frac{1}{2}}.
\end{equation}
This question is closely related to
average $5$-parts of class numbers of quadratic fields. In general, let $\ell$ be a prime, $D$ range over fundamental discriminants, and
$r_D:=\operatorname{rk}_{\ell}(\operatorname{Cl}_{\mathbb{Q}(\sqrt{D})}).$ Then the heuristics of Cohen-Lenstra predicts that
the average of $\ell^{r_D}-1$ over all imaginary quadratic fields
is $1$, and the average of $\ell^{r_D}-1$ over all real quadratic
fields is $\ell^{-1}$. 

\indent
In fact, one can show using class field theory  that the Cohen-Lenstra heuristics imply that Malle's conjecture is true for $D_5$ quintic fields. 
Conversely, the best known upper bound for $N(5,D_5,X)$ is proved using the ``trivial''
bound (\cite{Kluners-Dihed}):
\begin{equation}
\ell^{r_D}\leq\#\operatorname{Cl}_{\mathbb{Q}(\sqrt{D})}=O(D^{\frac{1}{2}}\log D).
\end{equation}
This gives $N(5,D_5,X)\ll X^{\frac{3}{4}+\varepsilon}$,
and any improved bound would give non-trivial information on average
$5$-parts of class groups in a similar manner.

In this paper, we
consider a method of point counting on varieties to give upper bounds
on $N(5,D_5,X)$. Our main result is the following:
\begin{thm} \label{mainthm}
To any quintic number field $K$ with Galois group $D_5$,
there corresponds a triple $(A, B, C)$ with $A, B \in \oo_{\qq[\sqrt{5}]}$
and $C \in \zz$, such that
\begin{equation}\nm^{\qq[\sqrt{5}]}_\qq \left(B^2 - 4 \cdot \bar{A} \cdot A^2\right) = 5 \cdot C^2 \end{equation}
and which satisfies the following under any archimedean valuation:
\begin{equation}|A| \ll D_K^{\frac{1}{4}}, \quad |B| \ll D_K^{\frac{3}{8}}, \quad \text{and} \quad |C| \ll D_K^{\frac{3}{4}}.\end{equation}
Conversely, the triple $(A,B,C)$ uniquely determines $K$.
\end{thm}

In Section~\ref{sec:num}, we further provide numerical evidence that
$N(5,D_5,X)\ll X^{\frac{2}{3}+\alpha}$ for very small $\alpha$; in particular the exponent appears to be much lower than $\frac{3}{4}$. 
\\ \indent 
Before we prove Theorem \ref{mainthm}, we show that earlier work of Wong \cite{Wong} in the case of $G=A_4$ can be handled in a similar fashion. Namely, we give a shorter proof of the following theorem:
\begin{thm}[Wong]\label{WongThm}
To any quartic number field $K$ with Galois group $A_4$, there corresponds a tuple $(a_2,a_3,a_4,y)\in\mathbb{Z}^4$, such that \[(4a_2^2 + 48a_4)^3 = \nm^{\qq[\sqrt{-3}]}_{\qq} \left(32 a_2^3 + 108 a_3^2 - 6 a_2 (4a_2^2 + 48 a_4) - 12\sqrt{-3} y\right),\] and which satisfies the following under any archimedean valuation:
\[|a_2|\ll D_K^{\frac{1}{3}},\ |a_3|\ll D_K^{\frac{1}{2}},\ |a_4|\ll D_K^{\frac{2}{3}},\text{ and } |y|\ll D_K.\] Conversely, given such a tuple, there corresponds at most one $A_4$-quartic field. In particular, we have that $N(4,A_4,X)\ll X^{\frac{5}{6}+\varepsilon}$.
\end{thm}
\section{Upper Bounds via Point Counting}\label{sec:varieties}

Let $G$ be a transitive permutation group.
If $K$ is a number field  of discriminant $D_K$ and degree $n$ for which
$\gal(\kgal / \qq) \simeq G$,
then Minkowski theory implies there is an element
$\alpha \in \oo_K$ of trace zero with
\[|\alpha| \ll D_K^{\frac{1}{2(n - 1)}} \quad \text{(under any archimedean valuation),}\]
where the implied constant depends only on $n$.
In particular, if $K$ is a primitive extension of $\qq$,
then $K = \qq(\alpha)$, so the characteristic polynomial
of $\alpha$ will determine $K$.
One can use this to give an upper bound on $N(n, G, X)$
(at least in the case where $K$ is primitive), since
every pair $(K, \alpha)$ as above gives a $\zz$-point
of $\spec \qq[x_1, x_2, \ldots, x_n]^G / (s_1)$,
where $s_1 = x_1 + x_2 + \cdots + x_n$ (here $\qq[x_1,x_2,\ldots,x_n]^G$ denotes the ring of $G$-invariant polynomials in $\qq[x_1,x_2,\ldots,x_n]$).

\section{Proof of Theorem \ref{WongThm}}

In this section, we sketch a simplified
(although essentially equivalent) version
of Wong's proof \cite{Wong} that
$N(4, A_4, X)\ll X^{\frac{5}{6} + \epsilon}$ as motivation for our main theorem. In this section,
we assume that the reader is familiar with the arguments in Wong's paper.
As noted in the last section, it suffices to count triples $(a_2, a_3, a_4)$
for which $|a_k| \ll X^{\frac{k}{6}}$ under any archimedean valuation and
\[256 a_4^3 - 128 a_2^2 a_4^2 + (16 a_2^4 + 144 a_2 a_3^2) a_4 - 4 a_2^3 a_3^2 - 27 a_3^4 = \disc(x^4 + a_2 x^2 + a_3 x + a_4) = y^2\]
for some $y\in\zz$. (See equation~4.2 of \cite{Wong}.)

The key
observation of Wong's paper (although he
does not state it in this way)
is that this equation can be rearranged as
\begin{equation} \label{wongeqn}
(4a_2^2 + 48a_4)^3 = \nm^{\qq[\sqrt{-3}]}_{\qq} \left(32 a_2^3 + 108 a_3^2 - 6 a_2 (4a_2^2 + 48 a_4) - 12\sqrt{-3} y\right).
\end{equation}

One now notes that there are $\ll X^{\frac{2}{3}}$ possibilities
for $4a_2^2 + 48a_4$, and for each of these choices
$(4a_2^2 + 48a_4)^3$ can be written in $\ll X^\varepsilon$
ways as a norm of an element of $\qq[\sqrt{-3}]$.
Thus, it suffices to count the number of points
$(a_2, a_3)$ for which
\[32 a_2^3 + 108 a_3^2 - 6 a_2 (4a_2^2 + 48 a_4) - 12\sqrt{-3} y \quad \text{and} \quad 4a_2^2 + 48 a_4\]
are fixed. But the above equation defines
an elliptic curve, on which the number of integral points
can be bounded by Theorem 3 in \cite{Heath-Brown}. This then gives Wong's bound (as well as the conditional bound assuming standard conjectures as Wong shows).

\section{Proof of Theorem~\ref{mainthm}}

In this section, we give the proof of Theorem~\ref{mainthm}.
As explained in Section~\ref{sec:varieties},
it suffices to understand the $\zz$-points of
\[\spec \qq[x_1, x_2, x_3, x_4, x_5]^{D_5} / (x_1 + x_2 + x_3 + x_4 + x_5)\]
inside a particular box. Write $\zeta$ for a primitive
$5$th root of unity, and define
\[V_j = \sum_{i = 1}^5 \zeta^{ij} x_i.\]
In terms of the $V_j$, we define
\begin{align*}
A &= V_2 \cdot V_3 \\
B &= V_1 \cdot V_2^2 + V_3^2 \cdot V_4 \\
C &= \frac{1}{\sqrt{5}} \cdot (V_1 \cdot V_2^2 - V_3^2 \cdot V_4) \cdot (V_2 \cdot V_4^2 - V_1^2 \cdot V_3).
\end{align*}

\begin{lm} The expressions $A$, $B$, and $C$ are invariant
under the action of $D_5$.
\end{lm}
\begin{proof}
Note that the generators of $D_5$ act by $V_j \mapsto V_{5 - j}$
and $V_j \mapsto \zeta^j V_j$; the result follows immediately.
\end{proof}

\begin{lm} We have $A, B \in \oo_{\qq[\sqrt{5}]}$ and $C \in \zz$.
\end{lm}

\begin{proof}
To see the first assertion, it suffices to show that
$A$ and $B$ are invariant by the element of $\gal(\qq[\zeta] / \qq)$
given by $\zeta \mapsto \zeta^{-1}$. But this induces
the map $V_j \mapsto V_{5 - j}$, so this is clear.

To see that $C$ is in $\zz$, we observe that the generator
of $\gal(\qq[\zeta] / \qq)$ given by $\zeta \mapsto \zeta^2$
acts by $C \sqrt{5} \mapsto -C \sqrt{5}$.
Since $C\sqrt{5}$ is an algebraic integer,
it follows that $C\sqrt{5}$ must be a rational integer times $\sqrt{5}$,
so $C \in \zz$.
\end{proof}

\noindent
Now, we compute
\[B^2 - 4 \cdot \bar{A} \cdot A^2 = (V_1 \cdot V_2^2 + V_3^2 \cdot V_4)^2 - 4 \cdot V_1 \cdot V_4 \cdot (V_2 \cdot V_3)^2 = (V_1 \cdot V_2^2 - V_3^2 \cdot V_4)^2.\]
Therefore,
\[\nm^{\qq[\sqrt{5}]}_\qq \left(B^2 - 4 \cdot \bar{A} \cdot A^2\right) = (V_1 \cdot V_2^2 - V_3^2 \cdot V_4)^2 \cdot (V_2 \cdot V_4^2 - V_1^2 \cdot V_3)^2 = 5 \cdot C^2,\]
which verifies the identity claimed in Theorem~\ref{mainthm}.

To finish the proof of Theorem~\ref{mainthm}, it remains
to show that to each triple $(A, B, C)$, there corresponds
at most one $D_5$-quintic field. To do this, we begin
with the following lemma.

\begin{lm} \label{nonzero} None of the $V_j$ are zero.
\end{lm}
\begin{proof} Suppose that some $V_j$ is zero.
Since $\bar{A} \cdot A^2 = V_1 \cdot V_2^2 \cdot V_3^2 \cdot V_4$, it follows that
$\bar{A} \cdot A^2 = 0$, and hence
\[\nm^{\qq[\sqrt{5}]}_\qq \left( B^2 \right) = 5 \cdot C^2 \quad \Rightarrow \quad B = C = 0.\]
Using $B = 0$, we have
$V_1 V_2^2 \cdot V_3^2 V_4 = V_1 V_2^2 + V_3^2 V_4 = 0$,
so $V_1V_2^2 = V_3^2 V_4 = 0$.
Similarly, using $\bar{B} = 0$, we have 
$V_2 V_4^2 = V_1^2 V_3 = 0$.
Thus, all pairwise products $V_i V_j$ with $i \neq j$
are zero, so at most one $V_k$ is nonzero.
Solving for the $x_i$, we find $x_i = \zeta^{-ik} c$ for some
constant $c$. (It is easy to verify that this is a solution,
since $\sum \zeta^i = 0$; it is unique up to rescaling because
the transformation $(x_i) \mapsto (V_i)$ is given by a Vandermonde
matrix of rank $4$).
Hence, the minimal polynomial of $\alpha$
is $t^5 - c^5 = 0$, which is visibly not a $D_5$ extension.
\end{proof}

\begin{lm} \label{two} For fixed $(A, B, C)$, there are at most
two possibilities for the ordered quadruple
\[\left(V_1 V_2^2, V_3^2 V_4, V_2 V_4^2, V_1^2 V_3\right).\]
\end{lm}
\begin{proof} Since
$V_1 V_2^2 + V_3^2 V_4 = B$ and
$V_1 V_2^2 \cdot V_3^2 V_4 = \bar{A} \cdot A^2$
are determined, there are at most two possibilities
for the ordered pair
$(V_1 V_2^2, V_3^2 V_4)$.
Similarly, there at most two possibilities
for the ordered pair $(V_2 V_4^2, V_1^2 V_3)$;
thus if $V_1 V_2^2 = V_3^2 V_4$, then we are done.
Otherwise,
\[V_2 \cdot V_4^2 - V_1^2 \cdot V_3 = \frac{C\sqrt{5}}{V_1 \cdot V_2^2 - V_3^2 \cdot V_4}.\]
Since 
$V_2 V_4^2 + V_1^2 V_3 = \bar{B}$,
this shows that the ordered pair
$(V_1 V_2^2, V_3^2 V_4)$ determines
$(V_2 V_4^2, V_1^2 V_3)$. Hence there are at most
two possibilities our ordered quadruple.
\end{proof}

\begin{lm}  \label{ten} For fixed $(A, B, C)$, there are at most
ten possibilities for $(V_1, V_2, V_3, V_4)$.
\end{lm}

\begin{proof} In light of Lemmas~\ref{two} and~\ref{nonzero},
it suffices to show there at most five possibilities
for $(V_1, V_2, V_3, V_4)$ when we have fixed
nonzero values for
\[(V_1 V_4, V_2 V_3, V_1 V_2^2, V_3^2 V_4, V_2 V_4^2, V_1^2 V_3).\]
But this follows from the identities
\begin{align*}
V_1^5 &= \frac{V_1 V_2^2 \cdot (V_1^2 V_3)^2}{(V_2 V_3)^2} \quad V_3 = \frac{V_1^2 V_3}{V_1^2} \quad V_4 = \frac{V_3^2 V_4}{V_3^2} \quad V_2 = \frac{V_2 V_4^2}{V_4^2}. && \qedhere
\end{align*}
\end{proof}

This completes the proof of Theorem~\ref{mainthm},
because $|D_5| = 10$, so each $D_5$-quintic field corresponds
to ten ordered quadruples $(V_1, V_2, V_3, V_4)$, each of which
can be seen to correspond to the same triple $(A, B, C)$.
Thus, the triple $(A, B, C)$ uniquely determines the $D_5$-quintic field,
since otherwise we would have at least $20$ quadruples $(V_1, V_2, V_3, V_4)$
corresponding to $(A, B, C)$, contradicting Lemma~\ref{ten}.

\section{The Quadratic Subfield}

\begin{prop} Suppose that $K$ is a $D_5$-quintic field
corresponding to a triple $(A, B, C)$ with $C \neq 0$.
Then the composite of $\qq[\sqrt{5}]$
with the unique quadratic subfield $F \subset \kgal$
is generated by adjoining to $\qq[\sqrt{5}]$ the square root of
\[(2\sqrt{5} - 10) \cdot (B^2 - 4 \cdot \bar{A} \cdot A^2).\]
\end{prop}

\begin{proof}
Using the results of the previous section, we note that
\[\sqrt{(2\sqrt{5} - 10) \cdot (B^2 - 4 \cdot \bar{A} \cdot A^2 )} = 2 \cdot (\zeta - \zeta^{-1}) \cdot (V_1 \cdot V_2^2 - V_3^2 \cdot V_4).\]
By inspection, the $D_5$-action on the above expression
is by the sign representation, and the
action of $\gal(\qq[\zeta] / \qq[\sqrt{5}])$
is trivial.
Hence, adjoining the above quantity to $\qq[\sqrt{5}]$
generates the composite of $\qq[\sqrt{5}]$ with the quadratic subfield
$F$.
\end{proof}

\section{Discussion of Computational Results \label{sec:num}}

Numerical evidence indicates that the number of triples
$(A, B, C)$ satisfying the conditions of Theorem~\ref{mainthm} is $O(X^{\frac{2}{3}+\alpha})$ for a small number $\alpha$
(in particular, much less than $O(X^{\frac{3}{4}})$).
More precisely, we have the
following table of results. The computation took approximately four hours on a 3.3 GHz CPU, using the program available
at \url{http://web.mit.edu/~elarson3/www/d5-count.py}.

\smallskip

\begin{center}
\begin{tabular}{c|c}
$X$ & $\#(A, B, C)$ \\ \hline
10 & 3 \\
31 & 3 \\
100 & 7 \\
316 & 55 \\
1000 & 127 \\
3162 & 397 \\
10000 & 951 \\
31622 & 2143 \\
100000 & 5145 \\
316227 & 11385 \\
1000000 & 25807 \\
3162277 & 57079 \\
\end{tabular}
\end{center}

\smallskip

The following log plot shows that after the first few data points, the least
squares best fit
to the last four data points given by
$y = 0.698 x + 0.506$
with slope a little more than $\frac{2}{3}$ is quite close.

\smallskip

\begin{center}
\includegraphics[width=0.5\textwidth]{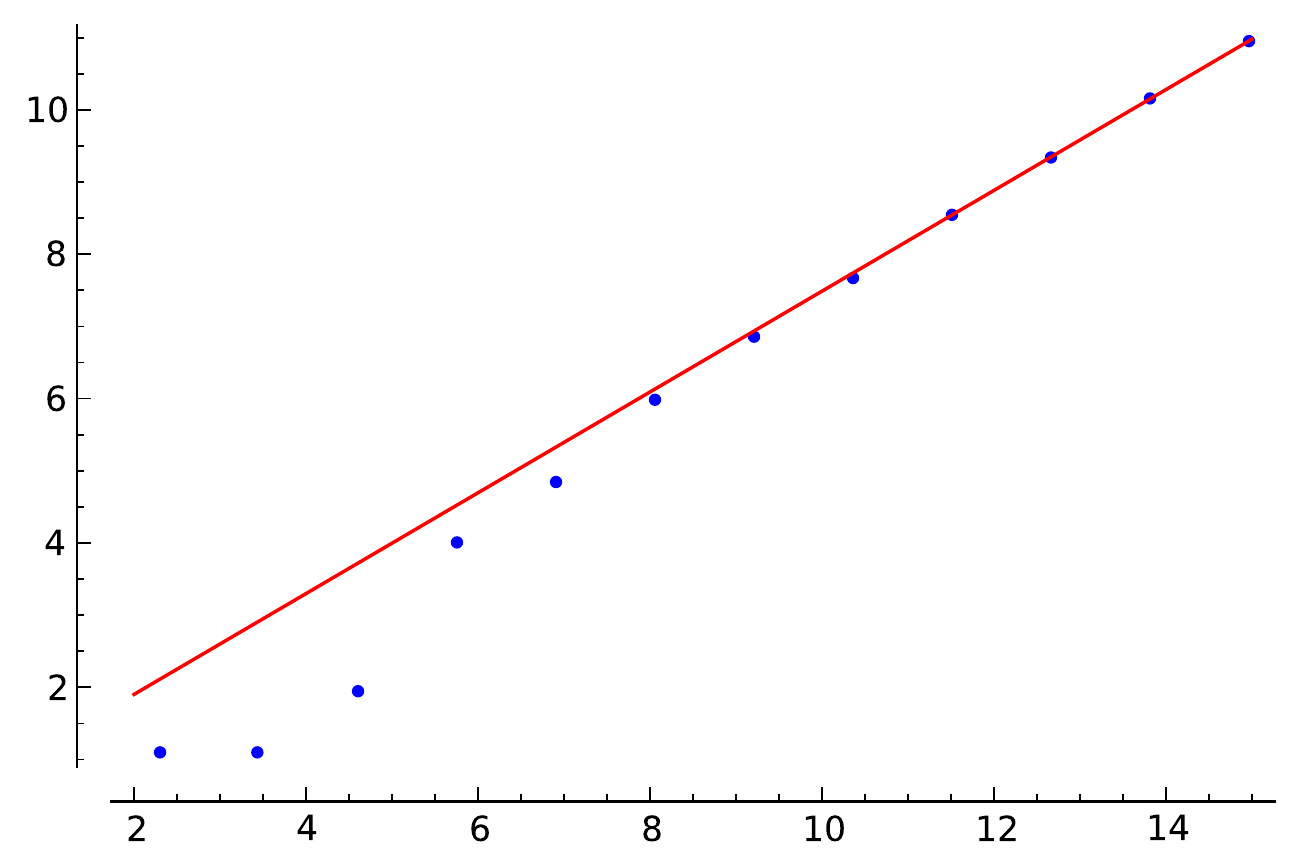}
\end{center}

\end{document}